%% file: IsogenyVolcanoes.tex
\documentclass{amsart}
\usepackage{amsmath,amssymb,amsthm,amscd,amsfonts}
\usepackage{enumerate,url,tikz,centernot}
\thanks{The author was supported by NSF grant DMS-1115455}

\newcommand{\Q}{\mathbb{Q}}
\newcommand{\C}{\mathbb{C}}
\renewcommand{\H}{\mathbb{H}}

\newcommand{\Z}{\mathbb{Z}}

\newcommand{\F}{\mathbb{F}}
\newcommand{\Fp}{\F_p}
\newcommand{\Fptwo}{\F_{p^2}}
\newcommand{\Fq}{\F_q}

\newcommand{\fa}{\mathfrak{a}}
\newcommand{\fb}{\mathfrak{b}}
\newcommand{\fl}{\mathfrak{l}}
\newcommand{\fp}{\mathfrak{p}}

\newcommand{\inkron}[2]{\genfrac {(}{)}{0.9pt}{}{#1}{#2}}
\newcommand{\M}{\textsf{M}}
\newcommand{\Ell}{{\rm Ell}}
\newcommand{\tr}{\operatorname{tr}}
\newcommand{\Aut}{\operatorname{Aut}}

\renewcommand{\vec}[1]{\boldsymbol{#1}}
\renewcommand{\O}{\mathcal{O}}
\def\disc{\operatorname{disc}}
\def\cl{\operatorname{cl}}
\def\Gal{\operatorname{Gal}}
\def\End{\operatorname{End}}
\def\iso{\xrightarrow{\sim}}
\def\ndiv{\centernot\mid}

\def\chr{\operatorname{char}}

\theoremstyle{plain}
\newtheorem{theorem}{Theorem}
\newtheorem{lemma}[theorem]{Lemma}
\newtheorem{proposition}[theorem]{Proposition}

\theoremstyle{definition}
\newtheorem{definition}[theorem]{Definition}
\newtheorem{remark}[theorem]{Remark}
\newtheorem{example}[theorem]{Example}

\begin{document}

\title{Isogeny volcanoes}
\author{Andrew V. Sutherland}
\address{Department of Mathematics\\Massachusetts Institute of Technology\\Cambridge, Massachusetts\ \  02139}
\email{drew@math.mit.edu}
%\subjclass[2000]{Primary 11G07,11Y16 ; Secondary  11G15, 11G20}

\begin{abstract}
The remarkable structure and computationally explicit form of isogeny graphs of elliptic curves over a finite field has made them an important tool for computational number theorists and practitioners of elliptic curve cryptography.
This expository paper recounts the theory behind these graphs and examines several recently developed algorithms that realize substantial (often dramatic) performance gains by exploiting this theory.
\end{abstract}
\maketitle
\section{Introduction}
A \emph{volcano} is a certain type of graph, one whose shape reminds us of the geological formation of the same name.   A typical volcano consists of a cycle with isomorphic balanced trees rooted at each vertex.
\medskip

\begin{figure}[htp]
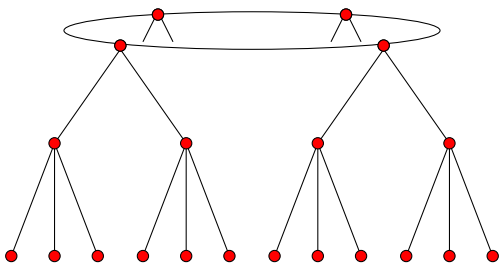

\include{VolcanoSideView}
\vspace{-10pt}
\caption{A volcano.}\label{figure:sideview}
\end{figure}

More formally, let $\ell$ be a prime.  We define an $\ell$-\emph{volcano} as follows.

\begin{definition}
An $\ell$-\emph{volcano} $V$ is a connected undirected graph whose vertices are partitioned into one or more \emph{levels} $V_0,\ldots,V_d$ such that the following hold:
\begin{enumerate}
\item
The subgraph on $V_0$ (the \emph{surface}) is a regular graph of degree at most 2.
\item
For $i>0$, each vertex in $V_i$ has exactly one neighbor in level $V_{i-1}$, and this accounts for every edge not on the surface.
\item
For $i<d$, each vertex in $V_i$ has degree $\ell+1$.
\end{enumerate}
\end{definition}

Self-loops and multi-edges are permitted in an $\ell$-volcano, but it follows from (ii) that these can only occur on the surface.  The integer $d$ is the \emph{depth} of the volcano (some authors use \emph{height}).  When $d=0$ only (i) applies, in which case $V$ is a connected regular graph of degree at most~$2$.
This is either a single vertex with up to two self-loops, two vertices connected by one or two edges, or a simple cycle on three or more vertices (the general case).
Figure~\ref{figure:topview} gives an overhead view of the volcano depicted in Figure~\ref{figure:sideview}, a 3-volcano of depth 2.

\begin{figure}[htp]
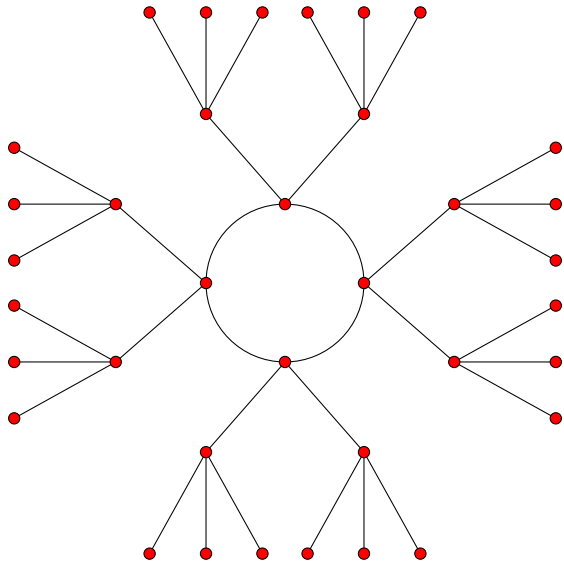

\include{VolcanoTopView}
\vspace{-10pt}
\caption{A 3-volcano of depth 2.}\label{figure:topview}
\end{figure}

We have defined volcanoes in purely graph-theoretic terms, but we are specifically interested in volcanoes that arise as components of graphs of isogenies between elliptic curves.
Our first objective is to understand how and why volcanoes arise in such graphs.  
The definitive work in this area was done by David Kohel, whose thesis explicates the structure of isogeny graphs of elliptic curves over finite fields~ \cite{Kohel:thesis}.
The term ``volcano'' came later, in work by Fouquet and Morain \cite{Fouquet:thesis,Fouquet:IsogenyVolcanoes} that popularized Kohel's work and gave one of the first examples of how isogeny volcanoes could be exploited by algorithms that work with elliptic curves.

This leads to our second objective: to show how isogeny volcanoes can be used to develop better algorithms.
We illustrate this with four examples of algorithms that use isogeny volcanoes to solve some standard computational problems related to elliptic curves over finite fields.
In each case, the isogeny-volcano approach yields a substantial practical and asymptotic improvement over the best previous results.

\section{Isogeny graphs of elliptic curves}
We begin by recalling some basic facts about elliptic curves and isogenies, all of which can be found in standard references such as \cite{Lang:EllipticFunctions,Silverman:EllipticCurves1,Silverman:EllipticCurves2}.

\subsection{Elliptic curves}
Let $k$ be a field.  An elliptic curve $E/k$ is a smooth projective curve of genus 1, together with a distinguished $k$-rational point~$0$.
If $k'/k$ is any field extension, the set $E(k')$ of $k'$-rational points of $E$ forms an abelian group with~$0$ as its identity element.
For convenience we assume that the characteristic of $k$ is not 2 or 3, in which case every elliptic curve $E/k$ can be defined as the projective closure of a short Weierstrass equation of the form
\[
Y^2 = X^3 + aX +b,
\]
where the coefficients $a,b\in k$ satisfy $4a^3+27b^2\ne 0$.
Distinct Weierstrass equations may define isomorphic curves:
the curves defined by $Y^2=X^3+a_1X+b_1$ and $Y^2=X^3+a_2X+b_2$ are isomorphic if and only if $a_2=u^4a_1$ and $b_2=u^6b_1$ for some $u$ (the isomorphism is then defined over the field $k(u)$).

Over the algebraic closure $\overline k$, we may identify the isomorphism class of an elliptic curve $E$ with its $j$-\emph{invariant}
\[
j(E)  = j(a,b) = 1728\frac{4a^3}{4a^3+27b^2},
\]
which does not depend on our choice of $a$ and $b$.
Note that while $j(E)$ lies in~$k$, it only determines the isomorphism class of $E$ over the algebraic closure $\overline k$.  Elliptic curves with the same $j$-invariant need not be isomorphic over $k$; such curves are said to be \emph{twists} of each other.
The $j$-invariants $j(0,b)=0$ and $j(a,0)=1728$ correspond to elliptic curves with extra automorphisms.
To simplify matters we will occasionally exclude these special cases from consideration.

Every $j\in k$ arises as the $j$-invariant of an elliptic curve $E/k$: the cases 0 and 1728 are addressed above, and otherwise if
\[
a = 3j(1728-j)\quad\text{and}\quad b= 2j(1728-j)^2,
\]
then $j=j(a,b)$.  There is thus a one-to-one correspondence between the field $k$ and the set of $\overline k$-isomorphism classes of elliptic curves.
This is the vertex set of the isogeny graphs that we wish to define.

\subsection{Isogenies}
An \emph{isogeny} $\varphi\colon E_1\to E_2$ is a morphism of elliptic curves, a rational map that preserves the identity.
Every nonzero isogeny induces a surjective group homomorphism from $E_1(\overline k)$ to $E_2(\overline k)$ that has a finite kernel.
Elliptic curves related by a nonzero isogeny are said to be \emph{isogenous}.

The \emph{degree} of a nonzero isogeny is its degree as a rational map (the zero isogeny has degree 0).
We call an isogeny of positive degree $n$ an $n$-\emph{isogeny}.
The kernel of an $n$-isogeny typically has cardinality $n$ (such isogenies are said to be \emph{separable}), and this is always the case when $n$ is not divisible by the characteristic of~$k$.
 We are primarily interested in isogenies of prime degree $\ell\ne\chr k$, 
and we shall only distinguish isogenies up to isomorphism, regarding isogenies $\phi$ and $\varphi$ as equivalent if $\phi=\iota_1\circ\varphi\circ\iota_2$ for some isomorphisms $\iota_1$ and $\iota_2$.

There are two important facts about isogenies that we need.
The first is that every finite subgroup of $E_1(\overline k)$ is the kernel of a separable isogeny that is uniquely determined (up to isomorphism)  \cite[Prop.~III.4.12]{Silverman:EllipticCurves1}, and this isogeny can be explicitly computed using V\'elu's algorithm \cite{Velu:Isogenies}.
The second is that every $n$-isogeny $\varphi\colon E_1\to E_2$ has a unique \emph{dual isogeny} $\hat\varphi\colon E_2\to E_1$ that~satisfies
\[
\varphi\circ\hat\varphi = \hat\varphi\circ\varphi = [n],
\]
where $[n]$ is  the \emph{multiplication-by}-$n$ \emph{map} that sends $P\in E_1(\overline k)$ to  $nP=P+\cdots+P$; see \cite[Thm.~III.6.1]{Silverman:EllipticCurves1}.
The dual isogeny $\hat\varphi$ has degree $n$, and $[n]$ has degree~$n^2$.

The kernel of the multiplication-by-$n$ map is the $n$-\emph{torsion subgroup}
\[
E[n] = \{ P\in E(\overline k): nP = 0\},
\]
and for $n$ not divisible by the characteristic of $k$ we have
\[
E[n]\simeq \Z/n\Z \times \Z/n\Z.
\]
For primes $\ell\ne\chr k$, there are $\ell+1$ cyclic subgroups in $E[\ell]$ of order $\ell$, each of which is the kernel of a separable $\ell$-isogeny.
Every $\ell$-isogeny $\varphi$ from $E$ arises in this way, since any point in the kernel of $\varphi$ also lies in the kernel of $\hat\varphi\circ\varphi=[\ell]$.

Not every $\ell$-isogeny $\varphi\colon E\to\tilde{E}$ is necessarily defined over $k$; this occurs precisely when $\ker\varphi$ is invariant under the action of the Galois group $G=\Gal(k(E[\ell])/k)$.
The Galois group acts linearly on $E[\ell]\simeq \Z/\ell\Z \times \Z/\ell\Z$, which we may view as an $\F_\ell$-vector space of dimension two in which the order~$\ell$ subgroups of $E[\ell]$ are linear subspaces.
If $G$ fixes more than two linear subspaces of a two-dimensional vector space then it must fix all of them.
This yields the following lemma.

\begin{lemma}
Let $E/k$ be an elliptic curve with $j(E)\ne 0,1728$ and and let $\ell\ne\chr k$ be a prime.
Up to isomorphism, the number of $k$-rational $\ell$-isogenies from $E$ is $0,1,2$, or $\ell+1$.
\end{lemma}

\subsection{The modular equation}
Let $j(\tau)$ be the classical modular function defined on the upper half plane $\H$.
For any $\tau\in\H$, the complex numbers $j(\tau)$ and $j(N\tau)$ are the $j$-invariants of elliptic curves defined over $\C$ that are related by an isogeny whose kernel is a cyclic group of order $N$.
The minimal polynomial $\Phi_N(Y)$ of the function $j(Nz)$ over the field $\C(j(z))$ has coefficients that are integer polynomials in~$j(z)$.
If we replace $j(z)$ with $X$ we obtain the \emph{modular polynomial} $\Phi_N\in\Z[X,Y]$, which is symmetric in $X$ and $Y$ and has degree $\ell+1$ in both variables.  It parameterizes pairs of elliptic curves over $\C$ related by a cyclic $N$-isogeny;
the \emph{modular equation} $\Phi_N(X,Y)=0$ is a canonical equation for the modular curve $Y_0(N)=\Gamma_0(N)\backslash\H$.

When $N$ is a prime $\ell$, every $N$-isogeny is cyclic, and we have
\begin{equation}\label{modeq}
\Phi_\ell\bigl(j(E_1),j(E_2)\bigr) = 0 \quad \Longleftrightarrow \quad j(E_1) \text{ and } j(E_2) \text{ are $\ell$-isogenous}.
\end{equation}
This moduli interpretation remains valid over any field.
On the RHS of (\ref{modeq}) we use $j(E_i)$ to denote the $\overline k$-isomorphism class of~$E_i$, and when we say that $j(E_1)$ and $j(E_2)$ are $\ell$-isogenous we mean that one can choose $\ell$-isogenous representatives $E_1$ and $E_2$ defined over $k$.
Over~$\C$, the choice of representatives does not matter, but over a non-algebraically closed field such as a finite field, we must choose compatible twists.  In practice this is easy to do.

\subsection{The graph of $\ell$-isogenies}
We now use the modular equation to define the graph of $\ell$-isogenies over a field $k$ of characteristic different from $\ell$.
\begin{definition}
The \emph{$\ell$-isogeny graph} $G_\ell(k)$ has vertex set $k$ and directed edges $(j_1,j_2)$ present with multiplicity equal to the multiplicity of $j_2$ as a root of $\Phi_\ell(j_1,Y)$. 
\end{definition}
The vertices of $G_\ell(k)$ are $j$-invariants, and its edges correspond to (isomorphism classes of) $\ell$-isogenies. 
Edges $(j_1,j_2)$ not incident to $0$ or $1728$ occur with the same multiplicity as $(j_2,j_2)$.
Thus the subgraph of $G_\ell(k)$ on $k\backslash\{0,1728\}$ is bi-directed, and we may view it as an undirected graph.
For any fixed $k$, the graphs $G_\ell(k)$ all have the same vertex set, but different edge sets, depending on $\ell$.
Given an elliptic curve $E/k$, we may view $j(E)$ as a vertex in any of these graphs, a fact that has many applications.

\subsection{Supersingular and ordinary components}
Over a field  of positive characteristic $p$, an elliptic curve is \emph{supersingular} if its $p$-torsion subgroup $E[p]$ is trivial; otherwise it is \emph{ordinary}.
If $E$ is supersingular, then so is any elliptic curve isogenous to $E$; it follows that $G_\ell(k)$ is composed of ordinary and supersingular components.

Every supersingular curve over $k$ can be defined over a quadratic extension of $k$'s prime field, thus every supersingular $j$-invariant in $\overline k$  lies in~$\Fptwo$ \cite[Thm.\ V.3.1]{Silverman:EllipticCurves1}.
It follows that if $E$ is supersingular, then the roots of $\Phi_\ell(j(E),Y)$ all lie in $\Fptwo$.
Thus the supersingular components of $G_\ell(\Fptwo)$ are regular graphs of degree $\ell+1$ (every vertex has out-degree $\ell+1$, vertices not adjacent or equal to $0$ or $1728$ also have in-degree $\ell+1$).

\begin{remark}[\textbf{Ramanujan graphs}]
In fact, $G_\ell(\Fptwo)$ has just one supersingular component \cite[Cor.\ 78]{Kohel:thesis}, and when $p\equiv 1\bmod 12$ it is a \emph{Ramanujan Graph} \cite{Pizer:RamanujanGraphs}, an expander graph with an essentially optimal expansion factor.  This has cryptographic applications~\cite{CGL:CryptographicExpanders}.
\end{remark}

We are primarily interested in the ordinary components of $G_\ell(k)$, since this is where we will find isogeny volcanoes.
First we need to recall some facts from the theory of complex multiplication.

\subsection{Complex multiplication}
An isogeny from an elliptic curve $E$ to itself is called an \emph{endomorphism}.
The endomorphisms of an elliptic curve $E$ form a ring $\End(E)$ in which addition and multiplication are defined via:
\[
(\phi+\varphi)(P) = \phi(P) + \varphi(P)\quad\text{ and }\quad(\phi\varphi)(P) = \phi(\varphi(P)).
\]
For any positive integer $n$, the multiplication-by-$n$ map $[n]$ lies in $\End(E)$, and we have $[n]\phi=\phi+\cdots+\phi=n\phi$ for all $\phi\in\End(E)$.
Since $[n]$ is never the zero endomorphism, it follows that $\End(E)$ contains a subring isomorphic to $\Z$, which we shall identify with $\Z$.

When $\End(E)$ is larger than $\Z$ we say that $E$ has \emph{complex multiplication} (CM), a term that arises from the fact that over the complex numbers, endomorphisms that do not lie in $\Z$ may be viewed as ``multiplication-by-$\alpha$'' maps for some algebraic integers $\alpha$.
Over a finite field $\Fq$, every elliptic curve has complex multiplication; for ordinary elliptic curves over $\Fq$, the Frobenius endomorphism that sends the point $(X,Y)$ to $(X^q,Y^q)$ is an example of an endomorphism that does not lie in $\Z$.

When $E$ has complex multiplication there are two possibilities:
\[
\End(E)\simeq \begin{cases}
\text{an order $\O$ in an imaginary quadratic field},\\
\text{an order $\O$ in a definite quaternion algebra},
\end{cases}
\]
and in either case we say that $E$ has CM by $\O$.
The second case occurs if and only if $E$ is supersingular, which is possible only in positive characteristic;
we are primarily interested in the first case.
It will be convenient to fix an isomorphism $\O\iso\End(E)$ so that we may regard elements of $\O$ as elements of $\End(E)$ and vice versa; this can be done canonically, as in \cite[Prop.~II.1.1]{Silverman:EllipticCurves2}.

The \emph{endomorphism algebra} $\End^0(E)=\End(E)\otimes\Q$ is preserved by nonzero isogenies.
Thus if $E$ has complex multiplication, then so does every elliptic curve isogenous~$E$, but not necessarily by the same order $\O$.

\subsection{Horizontal and vertical isogenies}
Let $\varphi\colon E_1\to E_2$ by an $\ell$-isogeny of elliptic curves with CM by imaginary quadratic orders $\O_1$ and $\O_2$ respectively.
Then $\O_1=\Z+\tau_1\Z$ and $\O_2=\Z+\tau_2\Z$, for some $\tau_1,\tau_2\in\H$.
The isogeny $\hat\varphi\circ\tau_2\circ\varphi$ lies in $\End(E_1)$, and this implies that $\ell\tau_2\in\O_1$;
similarly, $\ell\tau_1\in\O_2$.
There are thus three possibilities:
\begin{enumerate}
\item $\O_1=\O_2$, in which case $\varphi$ is \emph{horizontal};
\item $[\O_1:\O_2] = \ell$, in which case $\varphi$ is \emph{descending};
\item $[\O_2:\O_1] = \ell$, in which case $\varphi$ is \emph{ascending}.
\end{enumerate}
In the last two cases we say that $\varphi$ is a \emph{vertical} $\ell$-isogeny.
The orders $\O_1$ and $\O_2$ necessarily have the same fraction field $K=\End^0(E_1)=\End^0(E_2)$, and both lie in the maximal order $\O_K$, the ring of integers of $K$.

\subsection{The CM torsor}
Let $E/k$ be an elliptic curve with CM by an imaginary quadratic order $\O$, and let $\fa$ be an invertible $\O$-ideal.
The $\fa$-\emph{torsion subgroup}
\[
E[\fa] = \{P\in E(\overline k):\alpha(P)=0 \text{ for all } \alpha\in\fa\}
\]
is the kernel of a separable isogeny $\varphi_\fa\colon E\to E'$.
Provided that $\fa$ has norm not divisible by the characteristic of $k$, we have $\deg \varphi_\fa = N(\fa)=[\O:\fa]$.
Because $\fa$ is invertible, we must have $\End(E)\simeq \End(E')$; thus $\varphi_\fa$ is a horizontal isogeny.

If $\fa$ and $\fb$ are two invertible $\O$-ideals then $\varphi_{\fa\fb} = \varphi_{\fa}\varphi_{\fb}$.
Thus the group of invertible $\O$-ideals acts on the set of elliptic curves with endomorphism ring $\O$.
When $\fa$ is a principal ideal, we have $E\simeq E'$, thus there is an induced action of the ideal class group $\cl(\O)$ on the set
\[
\Ell_\O(k) = \{j(E): E/k \text{ with } \End(E)\simeq \O\}.
\]
This action is faithful (only principal ideals act trivially), and transitive (see \cite[Prop.~II.1.2]{Silverman:EllipticCurves2} for a proof in the case that $k=\C$ and $\O=\O_K$, which may be generalized via \cite[Ch.~10,13]{Lang:EllipticFunctions}).
Provided it is non-empty, the set $\Ell_\O(k)$ is thus a principal homogeneous space, a \emph{torsor}, for the group $\cl(\O)$.
The cardinality of $\Ell_\O(k)$ is either 0 or $h$, where $h=h(\O)=\#\cl(\O)$ is the \emph{class number}.  Thus either every curve $E/\overline k$ with CM by $\O$ is defined over $k$, or none of them are.

\begin{remark}[\textbf{Decomposing isogenies}]\label{remark:decomposing}
The  CM action allows us to express horizontal isogenies $\varphi_\fa$ of large degree as the composition of a sequence of isogenies of smaller degree.
Even if $\fa$ has prime norm, we may find that $[\fa]=[\fp_1\cdots\fp_s]$ in~$\cl(\O)$, where the~$\fp_i$ are prime ideals with norms smaller than $\fa$.
Under the generalized Riemann hypothesis (GRH), we can find, in probabilistic subexponential time, an equivalence $[\fa]=[\fp_1\cdots\fp_s]$ in which the $\fp_i$ have norms that are polylogarithmic in the class number $h$ and $s=O(\log h)$; see \cite[Thm.\ 2.1]{Childs:QuantumIsogenies}.
This makes horizontal isogenies asymptotically easier to compute than vertical isogenies (this holds even without the GRH), which has implications for cryptography;
see \cite{BCL:LargeIsogeny,Galbraith:Isogenies,Galbraith:GHSattack,GalbraithStolbunov:IsogenySmoothing, Jao:DLOG1,Jao:LargeIsogenies}.
\end{remark}

\subsection{Horizontal isogenies}
Every horizontal $\ell$-isogeny $\varphi$ arises from the action of an invertible $\O$-ideal $\fl$ of norm $\ell$, namely, the ideal of endomorphisms $\alpha\in\O$ whose kernels contain the kernel of $\varphi$.
If $\ell$ divides the index of $\O$ in the maximal order~$\O_K$ of its fraction field $K$, then no such ideals exist.
Otherwise we say that~$\O$ is \emph{maximal at} $\ell$, and there are then exactly
\[
1 + \left(\frac{\disc(K)}{\ell}\right) = \begin{cases}
0\qquad\text{$\ell$ is inert in $K$},\\
1\qquad\text{$\ell$ is ramified in $K$},\\
2\qquad\text{$\ell$ splits in $K$},
\end{cases}
\]
invertible $\O$-ideals of norm $\ell$, each of which gives rise to a horizontal $\ell$-isogeny.
In the split case we have $(\ell)=\fl\cdot\bar\fl$, and the $\fl$-orbits partition $\Ell_\O(k)$ into cycles corresponding to the cosets of $\langle[\fl]\rangle$ in $\cl(\O)$.
When $\fl$ is principal the ideal class $[\fl]$ is trivial, which leads to self-loops in $G_\ell(k)$.
We can also have $[\fl]=[\bar\fl]$ even though $\fl\ne\bar\fl$, which gives rise to double-edges in $G_\ell(k)$.

\subsection{Vertical isogenies}\label{subsection:vertical}
Let $\O$ be an imaginary quadratic order with discriminant $D$, and let $\O'=\Z+\ell\O$ be the order of index $\ell$ in $\O$.
To simplify matters, let us assume that $\O$ and $\O'$ have the same group of units $\{\pm 1\}$; this holds whenever $D<-4$, and excludes only the cases $\O=\Z[\zeta_3]$ and $\O=\Z[i]$, which correspond to the special $j$-invariants 0 and 1728 respectively.

The map that sends each invertible $\O'$-ideal $\fa$ to the invertible $\O$-ideal $\fa\O$ preserves norms and induces a surjective homomorphism
\[
\rho\colon\cl(\O')\to\cl(\O).
\]
See \cite[Prop.\ 7.20]{Cox:ComplexMultiplication} for a proof in the case that $\O$ is the maximal order; the general case is proved similarly (cf. \cite[Lem.\ 3]{BissonSutherland:Endomorphism} and \cite[\S 3]{BrokerLauterSutherland:CRTModPoly}).
Under a suitable identification of the class groups $\cl(\O')$ and $\cl(\O)$ with their torsors $\Ell_{\O'}(k)$ and $\Ell_{\O}(k)$, the vertical isogenies from $\Ell_{\O'}(k)$ to $\Ell_{\O}(k)$ correspond to the map from $\cl(\O')$ to $\cl(\O)$ given by $\rho$.
To show this, let us prove the following lemma.

\begin{lemma}\label{lemma:ascending}
Let $E'/k$ be an elliptic curve with CM by $\O'$.
Then there is a unique ascending $\ell$-isogeny from $E'$ to an elliptic curve $E/k$ with CM by $\O$.
\end{lemma}
\begin{proof}
The existence of $E'/k$ implies that $\Ell_{\O'}(k)$ is nonempty, and
since $\O$ contains $\O'$, it follows that $\Ell_{\O}(k)$ is also nonempty.\footnote{One way to see this is to note that $k$ contains all the roots of the Hilbert class polynomial for~$\O'$, hence it must contain all the roots of the Hilbert class polynomial for $\O$, since the ring class field of $\O'$ contains the ring class field of $\O$; see \S\ref{subsection:Hilbert}.}

Let us suppose that there exists an ascending $\ell$-isogeny $\phi_1\colon E_1'\to E_1$, for some elliptic curve $E_1'$ with CM by $\O'$.
Twisting $E_1$ if necessary, we may choose an invertible $\O'$-ideal $\fa'$ so that the horizontal isogeny $\varphi_{\fa'}$ maps $E_1'$ to $E'$.  If we now set $\fa=\rho(\fa')$ and let $E$ be the image of $\varphi_{\fa}\circ\phi_1$, then $E$ has CM by $\O$, and there is a unique isogeny $\phi\colon E'\to E$ such that $\phi\circ\varphi_{\fa'} = \varphi_{\fa}\circ\phi_1$, by \cite[Cor.\ 4.11]{Silverman:EllipticCurves1}.  We have $\deg \phi = \deg\varphi_{\fa}\deg\phi_1 / \deg\varphi_{\fa'} = \ell$, thus $\phi$ is an ascending $\ell$-isogeny.
It follows that if any elliptic curve $E_1'/k$ with CM by $\O'$ admits an ascending $\ell$-isogeny, then so does every such elliptic curve.

We now proceed by induction on $d=\nu_\ell([\O_K:\O])$.  Let $D_K=\disc(K)$.
For $d=0$, every elliptic curve $E/k$ with CM by $\O$ admits $\ell+1$ $k$-rational $\ell$-isogenies, of which $1+\inkron{D_K}{\ell}$ are horizontal.  The remaining $\ell-\inkron{D_K}{\ell} > 0$ must be descending, and their duals are ascending $\ell$-isogenies from elliptic curves with CM by~$\O'$.  It follows that there are a total of $(\ell-\inkron{D_K}{\ell})h(\O)$ ascending $\ell$-isogenies from $\Ell_{\O'}(k)$ to $\Ell_{\O}(k)$.  By  \cite[Thm.\ 7.24]{Cox:ComplexMultiplication}, this is equal to the cardinality $h(\O')$ of $\Ell_{\O'}(k)$.
Since there is at least one ascending $\ell$-isogeny from each elliptic curve $E'/k$ with CM by $\O'$, there must be exactly one in each case.

The argument for $d>0$ is similar.  By the inductive hypothesis, every elliptic curve $E/k$ with CM by $\O$ admits exactly one ascending $\ell$-isogeny, and since $\ell$ now divides $[\O_K:\O]$, there are no horizontal isogenies from $E$, and all $\ell$ of the remaining $\ell$-isogenies from $E$ must by descending.  There are thus a total of $\ell h(\O)$ ascending $\ell$-isogenies from $\Ell_{\O'}(k)$, which equals the cardinality $h(\O')$ of $\Ell_{\O'}(k)$.
\end{proof}

It follows from the proof of Lemma 5 that there is a one-to-one correspondence between the graph of the function $\rho$ and the edges of $G_\ell(k)$ that lead from $\Ell_{\O'}(k)$ to $\Ell_{\O}(k)$.  Indeed, let us pick a vertex $j_1'\in\Ell_{\O'}(k)$ and let $j_1$ be its unique neighbor in $\Ell_{\O}(k)$ given by Lemma~\ref{lemma:ascending}.  If we identify the edge $(j_1',j_1)$ in $G_\ell(k)$ with the edge $(1_{\cl(\O')},1_{\cl(\O)})$ in the graph of $\rho$, then every other edge in the correspondence is determined in a way that is compatible with the actions of $\cl(O')$ and $\cl(O)$ on the torsors $\Ell_{\O'}(k)$ and $\Ell_{\O}(k)$.
Under this correspondence, the vertices in $\Ell_{\O'}(k)$ that are connected to a given vertex $v$ in $\Ell_{\O}(k)$ (the \emph{children} of $v$) correspond to a coset of the kernel of $\rho$,  a cyclic group of order $\ell-\inkron{D_K}{\ell}$ generated by the class of an invertible $\O'$-ideal of norm~$\ell^2$; see \cite[Lem.\ 3.2]{BrokerLauterSutherland:CRTModPoly}.

\subsection{Ordinary elliptic curves over finite fields}
We now assume that $k$ is a finite field $\Fq$.
Let $E/\Fq$ be an ordinary elliptic curve and let $\pi_E$ denote the Frobenius endomorphism $(X,Y)\mapsto(X^q,Y^q)$.
The \emph{trace of Frobenius} is given by
\[
t = \tr \pi_E = q+1 - \#E(\Fq),
\]
and $\pi_E$ satisfies the characteristic equation $\pi_E^2-t\pi_E+q=0$.
As an element of the imaginary quadratic order $\O\simeq\End(E)$, the Frobenius endomorphism corresponds to an algebraic integer with trace $t$ and norm $q$.  Thus we have the \emph{norm equation}
\[ 
4q = t^2-v^2 D_K.
\]
in which $D_K$ is the discriminant of the field $K=\Q(\sqrt{t^2-4q})$ containing $\O$, and $v=[\O_K:\Z[\pi_E]]$.
We have
\[
\Z[\pi_E]\subseteq\O\subseteq\O_K,
\]
thus $[\O_K:\O]$ divides $v$, and the same is true for any elliptic curve $E/\Fq$ with Frobenius trace $t$.

Let us now define
\[
\Ell_t(\Fq) = \{j(E): E/\Fq \text{ satisfies } \tr\pi_E=t\},
\]
the set of $\overline\F_p$-isomorphism classes of elliptic curves over $\Fp$ with a given Frobenius trace $t$.
By a theorem of Tate \cite{Tate:IsogenyTheorem}, $\Ell_t(\Fq)$ corresponds to an isogeny class, but note that $\Ell_t(\Fq)=\Ell_{-t}(\Fq)$.  For any ordinary elliptic curve $E/\Fq$ with Frobenius trace $t=\tr \pi_E$, we may write $\Ell_t(\Fq)$ as the disjoint union
\[
\Ell_t(\Fq) = \bigsqcup_{\Z[\pi_E]\subseteq\O\subseteq\O_K}\Ell_\O(\Fq),
\]
of cardinality equal to the Kronecker class number $H(t^2-4q)$; see \cite[Def.\ 2.1]{Schoof:PlaneCubicCurves}.

\subsection{The main theorem}
We now arrive at our main theorem, which states that, except for the components of $0$ and $1728$, the ordinary components of $G_\ell(\Fq)$ are $\ell$-volcanoes, and precisely characterizes their structure.
The proof follows easily from the material we have presented, as  the reader may wish to verify.
\smallskip

\begin{theorem}[\textbf{Kohel}]\label{thm:Kohel}
Let $V$ be an ordinary component of $G_\ell(\Fq)$ that does not contain $0$ or $1728$.
Then $V$ is an $\ell$-volcano for which the following hold:
\begin{enumerate}[{\rm (i)}]
\item The vertices in level $V_i$ all have the same endomorphism ring $\O_i$.
\item The subgraph on $V_0$ has degree $1+\inkron{D_0}{\ell}$, where $D_0=\disc(\O_0)$.
\item If $\inkron{D_0}{\ell}\ge 0$, then $|V_0|$ is the order of $[\fl]$ in $\cl(\O_0)$; otherwise $|V_0|=1$.
\item The depth of $V$ is $d=\nu_\ell\left((t^2-4q)/D_0\right)/2$, where $t^2=(\tr\pi_E)^2$ for $j(E)\in V$.
\item $\ell\ndiv [\O_K:\O_0]$ and $[\O_i:\O_{i+1}]=\ell$ for $0\le i < d$.
\end{enumerate}
\end{theorem}
\smallskip

\begin{remark}[\textbf{Special cases}]\label{remark:special}
Theorem~\ref{thm:Kohel} is easily extended to the case where $V$ contains $0$ or $1728$.
Parts (i)-(v) still hold, the only necessary modification is the claim that $V$ is an $\ell$-volcano.
When $V$ contains $0$, if $V_1$ is non-empty then it contains $\frac{1}{3}\bigl(\ell-\inkron{-3}{\ell}\bigr)$ vertices, and each vertex in $V_1$ has three incoming edges from~$0$ but only one outgoing edge to $0$.
When $V$ contains $1728$, if $V_1$ is non-empty then it contains $\frac{1}{2}\bigl(\ell-\inkron{-1}{\ell}\bigr)$ vertices,  and each vertex in $V_1$ has two incoming edges from $1728$ but only one outgoing edge to $1728$.
This 3-to-1 (resp.\ 2-to-1) discrepancy arises from the action of $\Aut(E)$ on the cyclic subgroups of $E[\ell]$ when $j(E)=0$ (resp.\ 1728).
Otherwise, $V$ satisfies all the requirements of an $\ell$-volcano, and most of the algorithms we present in the next section are equally applicable to $V$.
\end{remark}

\begin{example}
Let $p=411751$ and $\ell=3$.
The graph $G_3(\Fp)$ has a total of $206254$ components, of which $205911$ are ordinary and $343$ are supersingular.
The supersingular components all lie in the same isogeny class (which is connected in $G_3(\Fptwo)$), while the ordinary components lie in $1283$ distinct isogeny classes.

Let us consider the isogeny class $\Ell_t(\Fp)$ for $t=52$.
We then have $4p=t^2-v^2D$ with $v=2\cdot 3^2\cdot 5$ and $D=-203$.
The subgraph $G_{\ell,t}(\Fp)$ of $G_\ell(\Fp)$ on $\Ell_t(\Fp)$ (also known as a \emph{cordillera} \cite{Miret:Cordillera}), consists of ten $3$-volcanoes, all of which have depth $d=\nu_\ell(v)=2$.
It contains a total $1008$ vertices distributed as follows:
\begin{itemize}
\item $648$ vertices lie in six 3-volcanoes with $[\O_K:\O_0]=10$ and $|V_0|=12$.
\item $216$ vertices lie in two 3-volcanoes with $[\O_K:\O_0]=5$ and $|V_0|=12$.
\item $108$ vertices lie in a  3-volcano with $[\O_K:\O_0]=2$ and $|V_0|=12$.
\item $36$ vertices lie in a 3-volcano with $[\O_K:\O_0]=1$ and $|V_0|=4$.
\end{itemize}
%Note that only in the last case is $\O_0=\O_K$, and only in the last two cases does $V_0$ contain all of $\Ell_{\O_0}(\Fp)$.
For comparison:
\begin{itemize}
\item $G_{2,52}(\Fp)$ consists of $252$ 2-volcanoes of depth 1 with $|V_0|=1$.
\item $G_{5,52}(\Fp)$ consists of $144$ 5-volcanoes of depth 1 with $|V_0|=1$.
\item $G_{7,52}(\Fp)$ consists of  $504$ 7-volcanoes with two vertices and one edge.
\item $G_{11,52}(\Fp)$ consists of $1008$ 11-volcanoes that are all isolated vertices.
\end{itemize}
\end{example}

\section{Applications}
We now consider several applications of isogeny volcanoes, starting with one that is very simple, but nevertheless instructive.

\subsection{Finding the floor}
Let $E/\Fq$ be an ordinary elliptic curve.
Then $j(E)$ lies in an ordinary component $V$ of $G_\ell(\Fq)$.
We wish to find a vertex on the \emph{floor} of $V$, that is, a vertex~$v$ in level $V_d$, where~$d$ is the depth of $V$.
Such vertices $v$ are easily distinguished by their (out-) degree, which is the number of roots of $\Phi_\ell(v,Y)$ that lie in $\Fq$ (counted with multiplicity).

\begin{proposition}
Let $v$ be a vertex in an ordinary component $V$ of depth $d$ in $G_\ell(\Fq)$.
Either  $\deg v\le 2$ and $v\in V_d$, or $\deg v = \ell+1$ and $v\not\in V_d$.
\end{proposition}
\begin{proof}
If $d=0$ then $V=V_0=V_d$ is a regular graph of degree at most 2 and $v\in V_d$.
Otherwise, either $v\in V_d$ and $v$ has degree 1, or $v\not\in V_d$ and $v$ has degree $\ell+1$.
\end{proof}

We note that if $j(E)$ is on the floor then $E[\ell](\Fq)$ is necessarily cyclic (otherwise there would be another level below the floor).
This is useful, for example, when using the CM method to construct Edwards curves \cite{Morain:EdwardsCM}, and shows that every ordinary elliptic curve $E/\Fq$ is isogenous to some $E'/\Fq$ with $E'(\Fq)$ cyclic.

Our strategy for finding the floor is simple: if $v_0=j(E)$ is not already on the floor then we will construct a random path from $v_0$ to a vertex $v_s$ on the floor.
By a \emph{path}, we mean a sequence of vertices $v_0,v_1,\ldots,v_s$ such that each pair $(v_{i-1},v_i)$ is an edge and $v_i\ne v_{i-2}$ (so backtracking is prohibited).
\medskip

\noindent
\textbf{Algorithm} \textsc{FindFloor}\\
Given an ordinary vertex $v_0\in G_\ell(\Fq)$, find a vertex on the floor of its component.
\begin{enumerate}[1.]
\item If $\deg v_0 \le 2$ then output $v_0$ and terminate.
\item Pick a random neighbor $v_1$ of $v_0$ and set $s\leftarrow 1$.
\item While $\deg v_s > 1$: pick a random neighbor $v_{s+1}\ne v_{s-1}$ of $v_s$ and increment~$s$.
\item Output $v_s$.
\end{enumerate}
\smallskip

The complexity of \textsc{FindFloor} is given by the following proposition, in which $\M(n)$ denotes the time to multiply two $n$-bit integers.  
It is worth noting that for large $\ell$ the complexity is dominated by the time to substitute $v$ into $\Phi_\ell(X,Y)$, not by root-finding (a fact that is occasionally overlooked).

\begin{proposition}\label{prop:findfloor}.
Given $\Phi_\ell\in\Fq[X,Y]$, each step of \textsc{FindFloor} can be accomplished in \emph{$O(\ell^2\M(n) + \M(\ell n)n)$} expected time, where $n=\log q$.  The expected number of steps $s$ is $\delta+O(1)$, where $\delta$ is the distance from $v_0$ to the floor.
\end{proposition}
\begin{proof}
Computing $\phi(Y)=\Phi_\ell(v,Y)$ involves $O(\ell^2)$ $\Fq$-operations, or $O(\ell^2\M(n))$ bit operations.
The neighbors of $v$ are the distinct roots of $\phi(Y)$ that lie in $\Fq$, which are precisely the roots of $f(Y)=\gcd(Y^q-Y,\phi(Y))$.
Computing $Y^q\bmod \phi$ involves $O(n)$ multiplications in the ring $\Fq[Y]/(\phi)$, each of which can be accomplished using $O(\M(\ell n))$ bit operations, via Kronecker substitution \cite{Gathen:ComputerAlgebra}, yielding an $O(\M(\ell n)n)$ bound.
With the fast Euclidean algorithm the gcd of two polynomials of degree $O(\ell)$ can be computed using $O(\M(\ell n)\log\ell)$ bit operations. 
If $\log\ell < n$ then this is bounded by $O(\M(\ell n)n)$, and otherwise it is bounded by $O(\ell^2\M(n))$.
Thus the total time to compute $f(Y)$ for any particular $v$ is $O(\ell^2\M(n) + \M(\ell n)n)$.

The degree of $f(Y)$ is the number of distinct roots of $\Phi_\ell(Y,v)$ in $\Fq$.
For $\ell > 3$, this is less than or equal to 2 if and only if $v$ is on the floor.
For $\ell\le 3$ we can count roots with multiplicity  by taking gcds with derivatives of $\phi$, within the same time bound.
To find a random root of $f(Y)$ we use the probabilistic splitting algorithm of \cite{Rabin:RootFinding}; since we need only one root, this takes $O(\M(\ell n)n)$ expected time.

For every vertex $v$ in a level $V_i$ above the floor, at least 1/3 of $v$'s neighbors lie in in level $V_{i+1}$, thus within $O(1)$ expected steps the path will be extended along a descending edge.
Once this occurs, every subsequent edge in the path must be descending, since we are not allowed to backtrack along the single ascending edge, and will reach the floor within $\delta+O(1)$ steps.
\end{proof}

\begin{remark}[\textbf{Removing known roots}]
As a minor optimization, rather than picking $v_{s+1}$ as a root of $\phi(Y)=\Phi_\ell(v_s,Y)$ in step 3 of the \textsc{FindFloor} algorithm, we may use $\phi(Y)/(Y-v_{s-1})^e$, where $e$ is the multiplicity of $v_{s-1}$ as a root of $\phi(Y)$.
This is slightly faster and eliminates the need to check that $v_{s+1}\ne v_{s-1}$.
\end{remark}

The \textsc{FindFloor} algorithm  finds a path of expected length $\delta+O(1)$ from $v_0$ to the floor.
With a bit more effort we can find a path of exactly length $\delta$, using a simplified version of an algorithm from \cite{Fouquet:IsogenyVolcanoes}.
\medskip

\noindent
\textbf{Algorithm} \textsc{FindShortestPathToFloor}\\
Given an ordinary $v_0\in G_\ell(\Fq)$, find a shortest path to the floor of its component.
\begin{enumerate}[1.]
\item Let $v_0=j(E)$.  If $\deg v_0 \le 2$ then output $v_0$ and terminate.
\item Pick three neighbors of $v_0$ and extend paths from each of these neighbors in parallel, stopping as soon as any of them reaches the floor.\footnote{If $v_0$ does not have three distinct neighbors then just pick all of them.} 
\item Output a path that reached the floor.
\end{enumerate}
\smallskip

The correctness of the algorithm follows from the fact that at most two of $v_0$'s neighbors do not lie along descending edges, so one of the three paths must begin with a descending edge.
This path must then consist entirely of descending edges, yielding a shortest path to the floor.
The algorithm takes at most $3\delta$ steps, each of which has complexity bounded as in Proposition~\ref{prop:findfloor}.

The main virtue of  \textsc{FindShortestPathToFloor} is that it allows us to compute $\delta$, which tells us the level $V_{d-\delta}$ of $j(E)$ relative to the floor $V_d$.
It effectively gives us an ``altimeter'' $\delta(v)$ that we may be used to navigate $V$.
We can determine whether a given edge $(v_1,v_2)$ is horizontal, ascending, or descending, by comparing $\delta(v_1)$ to $\delta(v_2)$, and we can determine the exact level of any vertex; see \cite[\S 4.1]{Sutherland:HilbertClassPolynomials} for algorithms and further details.
We should also mention that an alternative approach based on pairings has recently been developed by Ionica and Joux \cite{Ionica:PairingTheVolcano2,Ionica:PairingTheVolcano}, which is more efficient when $d$ is large.

\subsection{Identifying supersingular curves}
Both algorithms in the previous section assume that their input is the $j$-invariant of an ordinary elliptic curve.
But what if this is not the case?
If we attempt to ``find the floor'' on the supersingular component of $G_\ell(\Fptwo)$ we will never succeed, since every vertex has degree $\ell+1$.
On the other hand, from part (iv) of Theorem~\ref{thm:Kohel} (and Remark~\ref{remark:special}), we know that every ordinary component of $G_\ell(\Fptwo)$ has depth less than $\log_\ell 2p$, so we can bound the length of the shortest path to the floor from any ordinary vertex.

This suggests that, with minor modifications, the algorithm \textsc{FindShortestPathToFloor} can be used to determine whether a given elliptic curve $E/\Fq$ is ordinary or supersingular.
If $j(E)\not\in\Fptwo$ then $E$ must be ordinary, so we may assume $v_0=j(E)\in\Fptwo$ (even if $E$ is defined over $\Fp$, we want to work in $\Fptwo$).
We modify step 2 of the algorithm so that if none of the three paths reaches the floor within $\log_\ell 2p$ steps, it reports that its input is supersingular and terminates.
Otherwise, the algorithm succeeds and can report that its input is ordinary.
This works for any prime $\ell$, but using $\ell=2$ gives the best running time.

This yields a Las Vegas algorithm to determine whether a given elliptic curve is ordinary or supersingular in $\tilde{O}(n^3)$ expected time, where $n=\log q$.
For comparison, the best previously known Las Vegas algorithm has an expected running time of $\tilde{O}(n^4)$, and the best known deterministic algorithm runs in $\tilde{O}(n^5)$ time.
Remarkably, the average time for a random input is only $\tilde{O}(n^2)$.
This matches the complexity of the best known Monte Carlo algorithm for this problem, with better constant factors; see \cite{Sutherland:Supersingular} for further details.

\subsection{Computing endomorphism rings}
We now turn to a more difficult problem: determining the endomorphism ring of an ordinary elliptic curve $E/\Fq$.
We assume that the trace of Frobenius $t=\tr \pi_E$ is known; this can be computed in polynomial time using Schoof's algorithm \cite{Schoof:ECPointCounting1}.
By factoring $4q-t^2$, we can compute the positive integer $v$ and fundamental discriminant $D$ satisfying the norm equation $4q=t^2-v^2D$.
We then know that $\Z[\pi_E]$ has index $v$ in the maximal order $\O_K$, where $K=\Q(\sqrt{D})$.
The order $\O\simeq \End(E)$ is uniquely determined by its index $u$ in $\O_K$, and $u$ must be a divisor of $v$.
Let us assume $D<-4$.

We can determine $u$ by determining the level of $j(E)$ in its component of $G_\ell(\Fq)$ for each of the primes $\ell$ dividing $v$.
If $v=\ell_1^{e_1}\cdots\ell_w^{e_w}$ is the prime factorization of~$v$, then $u=\ell_1^{d_1}\cdots\ell_w^{d_w}$, where $\delta_i=e_i-d_i$ is the distance from $j(E)$ to the floor of its $\ell_i$-volcano.
But it may not be practical to compute $\delta_i$ using \textsc{FindShortestPathToFloor} when $\ell_i$ is large:
its complexity is quasi-quadratic in~$\ell_i$, which may be exponential in $\log q$ (and computing $\Phi_{\ell_i}$ is even harder).
More generally, we do not know any algorithm for computing a vertical $\ell$-isogeny whose complexity is not at least linear in $\ell$ (in general, quadratic in $\ell$).
This would seem to imply that we cannot avoid a running time that is exponential in $\log q$.

However, as noted in Remark~\ref{remark:decomposing}, computing horizontal isogenies is easier than computing vertical isogenies.
We now sketch an approach to computing $\End(E)$ that uses horizontal isogenies to handle large primes dividing $v$, based on the algorithm in~\cite{BissonSutherland:Endomorphism}.
To simplify the presentation, we assume that $v$ is square-free; the generalization to arbitrary $v$ is straight-forward.

Let $\mathcal{L}$ be the lattice of orders in $\O_K$ that contain $\Z[\pi_E]$.
Our strategy is to determine whether $u$ is divisible by a given prime divisor $\ell$ of $v$ using a smooth relation that holds in an order $\O\in \mathcal{L}$ if and only if $\O$ is maximal at $\ell$.
This relation will hold in $\End(E)$ if and only if $u$ is not divisible by $\ell$.

A \emph{smooth relation} $R$ is a multiset $\{\fp_1^{r_1}\cdots \fp_s^{r_s}\}$ in which the~$\fp_i$ are invertible $\Z[\pi_E]$-ideals with prime norms $p_i$ occurring with multiplicity $r_i$, such that $p_i$ and~$r_i$ satisfy bounds that are subexponential in $\log q$.
We say that $R$ \emph{holds} in $\O\in\mathcal{L}$ if the $\O$-ideal $R_\O=(\fp_1\O)^{r_1}\cdots(\fp_s\O)^{r_s}$ is principal.
If $\O'\subset\O$, the existence of the norm-preserving homomorphism $\rho\colon \cl(\O')\to\cl(\O)$ defined as in \S\ref{subsection:vertical} implies that if $R$ holds in $\O'$, then it holds in $\O$.
It thus suffices to find a relation that holds in the order of index $v/\ell$ in $\O_K$, but not in the order of index $\ell$ in $\O_K$.
Under the GRH, for $\ell>3$ we can find such an $R$ in probabilistic subexponential time \cite{Bisson:Endomorphism}.

To determine whether $R$ holds in $\O\simeq\End(E)$, we compute the CM action of $[R_\O]\in\cl(\O)$ on $j(E)\in\Ell_O(\Fq)$.
This involves walking $r_i$ steps along the surface of a $p_i$-volcano for each of the $\fp_i$ appearing in $R$ and then checking whether we wind up back at our starting point $j(E)$.
None of the $p_i$ divide $v$, so these $p_i$-volcanoes all have depth 0 and consist of either a single edge or a cycle.
We must choose a direction to walk along each cycle (one corresponds to the action of $\fp_i$, the other to~$\bar\fp_i$).
There are methods to determine the correct choice, but in practice we can make $s$ small enough so that it is easy to simply try every combination of choices and count how many work; see \cite{BissonSutherland:Endomorphism} for details.

Under the GRH, this algorithm has a subexponential expected running time of $L[1/2,\sqrt{3}/2]$ plus the cost of factoring $4q-t^2$ (the latter is heuristically negligible, using the number field sieve, and provably bounded by $L[1/2,1]$ in \cite{Lenstra:RigorousFactoring}).
Bisson  \cite{Bisson:Endomorphism} has recently improved this to $L[1/2,\sqrt{2}/2]$ plus the cost of factoring $4q-t^2$.
%For comparison, the best deterministic result is Kohel's algorithm \cite{Kohel:thesis}, which runs in $\tilde{O}(q^{1/3})$ time, under the GRH.

\begin{example}
Let $q=2^{320}+261$ and suppose that $E/\Fq$ has Frobenius trace
\[
t=2306414344576213633891236434392671392737040459558.
\]
Then $4q=t^2-v^2D$, where $D=-147759$ and $v=2^2p_1p_2$, with
\begin{align*}
p_1 &= 16447689059735824784039,\\
p_2 &= 71003976975490059472571.
\end{align*}
We can easily determine the level of $j(E)$ in its 2-volcano by finding a shortest path to the floor.
For $p_1$ and $p_2$ we instead use smooth relations $R_1$ and $R_2$.

Let $\O_1$ be the order of index $p_1$ in $\O_K$, and $\O_1'$ the order of index $v/p_1$ in~$\O_K$.
The relation
\[
R_1 = \{\fp_5, \fp_{19}^2, \bar\fp_{23}^{210},\fp_{29},\fp_{31},\bar\fp_{41}^{145},\fp_{139}, \bar\fp_{149},\fp_{167},\bar\fp_{191},\bar\fp_{251}^6,\fp_{269},\bar\fp_{587}^7,\bar\fp_{643}\}
\]
holds in $\O_1$ but not in $\O_1'$ (here $\fp_\ell$ denotes the ideal of norm $\ell$ corresponding to the reduced binary quadratic form $\ell x^2+bxy+cy^2$ with $b\ge 0$).
If we now let~$\O_2$ be the order of index $p_2$ in $\O_K$ and $\O_2'$ the order of index $v/p_2$ in $\O_K$, then
\[
R_2 = \{ \fp_{11},\bar\fp_{13}^{576},\fp_{23}^2,\bar\fp_{41},\bar\fp_{47},\fp_{83},\fp_{101},\bar\fp_{197}^{28},\bar\fp_{307}^3,\fp_{317},\bar\fp_{419},\fp_{911}\}
\]
holds in $\O_2$ but not in $\O_2'$.

% smoothrelation -11918985045661593147801358526441012387027596491329904 1000 0 0 1
% 5^1 19^2 23^-210 29^1 31^1 41^-145 139^1 149^-1 167^1 191^-1 251^-6 269^1 587^-7 643^-1
% seed 1 took 692.7 s
% zp_count_relation 2e320+261 18684687435969412112731976076100210226018622398620728744231729737446652637977030538492635400272 587^7 643^1 41^145 23^210 251^6 269^1 191^1 167^1 149^1 139^1  19^2 31^1 29^1 5^1
%1 * 587^7...45.350 secs
%2 * 643^1...43.490 secs
%4 * 41^145...42.400 secs
%8 * 23^210...56.710 secs
%16 * 251^6...77.820 secs
%32 * 269^1...16.860 secs
%64 * 191^1...20.000 secs
%128 * 167^1...32.450 secs
%256 * 149^1...54.860 secs
%512 * 139^1...98.110 secs
%1024 * 19^2...39.570 secs
%2048 * 31^1...53.580 secs
%4096 * 29^1...96.880 secs
%8192 * 5^1...19.070 secs
% #R/D_E = 66 computed in 697.150 secs
% smoothrelation -639563543671659797025077759658958091656261361751024 1000 0 0 1
% 11^1 13^-576 23^2 41^-1 47^-1 83^1 101^1 197^-28 307^-3 317^1 419^-1 911^1
% seed 1 took 9.4s
% zp_count_relation 2e320+261 18684687435969412112731976076100210226018622398620728744231729737446652637977030538492635400272 911^1 197^28 13^576 307^3 419^1 317^1 101^1 83^1 47^1 23^2 41^1 11^1
%1 * 911^1...136.110 secs
%2 * 197^28...34.600 secs
%4 * 13^576...34.780 secs
%8 * 307^3...26.170 secs
%16 * 419^1...23.280 secs
%32 * 317^1...22.040 secs
%64 * 101^1...8.470 secs
%128 * 83^1...12.830 secs
%256 * 47^1...11.890 secs
%512 * 23^2...27.010 secs
%1024 * 41^1...39.630 secs
%2048 * 11^1...14.050 secs
% #R/D_E = 34 computed in 390.870 secs
%

Including the time to compute the required modular polynomials and the time to find the relations $R_1$ and $R_2$, the total time to compute $\End(E)$ in this example is less than half an hour.
In contrast, it would be completely infeasible to directly compute a vertical isogeny of degree $p_1$ or $p_2$; writing down even a single element of the kernel of such an isogeny would require more than $2^{80}$ bits.
\end{example}

\subsection{Computing Hilbert class polynomials}\label{subsection:Hilbert}
Let $\O$ be an imaginary quadratic order with discriminant $D$.
The \emph{Hilbert class polynomial} $H_D$ is defined by
\[
H_D(X) = \prod_{j\in\Ell_\O(\C)} (X-j).
\]
Equivalently, $H_D(X)$ is the minimal polynomial of the $j$-invariant of the lattice~$\O$ over the field $K=\Q(\sqrt{D})$.
Remarkably, its coefficients lie in $\Z$.

The field $K_\O=K(j(\O))$ is the \emph{ring class field} of~$\O$.
If $q$ splits completely in~$K_\O$, then $H_D(X)$ splits completely in $\Fq[X]$ and its roots form the set $\Ell_\O(\Fq)$.
Each root is then the $j$-invariant of an elliptic curve $E/\Fq$ with $\End(E)\simeq \O$.
We must have $\#E(\Fq)=q+1-t$, where $t=\tr(\pi_E)$ is prime to $q$, and the norm equation $4q=t^2-v^2D$ then uniquely determines the integers $t$ and $v$ up to sign, for $D<-4$.
We can thus use a root of $H_D(X)$ in $\Fq$ to construct an elliptic curve $E/\Fq$ with exactly $q+1-t$ rational points.
Under reasonably heuristic assumptions \cite{Broker:EfficientCMConstructions}, 
by choosing $q$ and $D$ appropriately we can achieve any desired cardinality for $E(\Fq)$.
This is known as the \emph{CM method}, and is widely used in elliptic curve cryptography and elliptic curve primality proving.

We now outline an algorithm to compute $H_D(X)$ using the CRT approach described in \cite{Belding:HilbertClassPolynomial,Sutherland:HilbertClassPolynomials}.
Under the GRH it runs in $O(|D|(\log|D|)^{5+o(1)})$ expected time, which is quasi-linear in the $O(|D|\log|D|)$ size of $H_D(X)$.
%Heuristically, its running time is $O(|D|(\log|D|)^{3+o(1)})$, and in practice it is the fastest method known for computing $H_D(X)$.
The same approach can be used to compute many other types of class polynomials; see \cite{EngeSutherland:CRTClassInvariants}.
\medskip

\noindent
\textbf{Algorithm} \textsc{ComputeHilbertClassPolynomial}\\
Given an imaginary quadratic discriminant $D$, compute $H_D(X)$ as follows:
\begin{enumerate}[1.]
\item Select a sufficiently large set of primes $p$ that satisfy $4p=t^2-v^2D$.
\item For each prime $p$, compute $H_D(X)\bmod p$ as follows:
\begin{enumerate}[a.]
\item Generate random elliptic curves $E/\Fp$ until $\#E(\Fp)=p+1-t$.
\item Use volcano climbing to find $E'$ isogenous to $E$ with $\End(E')\simeq \O$.
\item Enumerate $\Ell_\O(\Fp)$ by applying the $\cl(\O)$-action to $j(E')$.
\item Compute $H_D(X) = \prod_{j\in\Ell_\O(\Fp)}(X-j) \bmod p$.
\end{enumerate}
\item Use the CRT to recover $H_D(X)$ over $\Z$ (or over $\Fq$, via the explicit CRT).
\end{enumerate}
\smallskip

Isogeny volcanoes play a key role in the efficient implementation of this algorithm, not only in step 2b, but also in step 2c, which is the most critical step and merits further discussion.  Given any sequence of generators $\alpha_1,\ldots,\alpha_k$ for a finite abelian group $G$, if we let $G_i=\langle \alpha_1,\ldots,\alpha_i\rangle$ and define $r_i = [G_i:G_{i-1}]$, then every element $\beta$ of $G$ can be uniquely represented in the form $\beta = \alpha_1^{e_1}\cdots \alpha_k^{e_k}$, with $0\le e_i < r_i$.
This is a special case of a \emph{polycyclic presentation}.
We can use a polycyclic presentation of $\cl(\O)$ to enumerate the torsor $\Ell_\O(\Fp)$ by
enumerating the list of exponent vectors $(e_1,\ldots,e_k)$ in reverse lexicographic order.
At each step we apply the action of the generator $\alpha_i$ that transforms the current exponent vector to the next in the list (usually $i=1$, since $e_1$ varies most frequently).

Using generators of the form $\alpha_i=[\fl_i]$, where $\fl_i$ is an invertible $\O$-ideal of prime norm $\ell_i$, this amounts to walking along the surfaces of various $\ell$-volcanoes.
To make this process as efficient as possible, it is crucial to minimize the primes $\ell_i$.
This is achieved by choosing $\fl_1$ to minimize $\ell_1$ and then minimizing each $\ell_i$ subject to $[\fl_i]\not\in \langle [\fl_1],\ldots,[\fl_{i-1}]\rangle$; this is called an \emph{optimal presentation} \cite[\S 5.1]{Sutherland:HilbertClassPolynomials}.
This will often cause us to use a set of generators that is larger than strictly necessary.

As an example, for $D=-79947$ the class group $\cl(\O)$ is cyclic of order 100, generated by the class of an ideal with norm 19.
But the optimal presentation for $\cl(\O)$ uses ideals $\fl_1$ and $\fl_2$ with norms 2 and 13, respectively.
The classes of these ideals are not independent, we have $[\fl_2]^{5} = [\fl_1]^{18}$, but they do form a polycyclic presentation with $r_1=20$ and $r_2=5$.  Using this presentation to enumerate $\Ell_\O(\Fp)$ is more than 100 times faster than using any single generator of $\cl(\O)$.  One can construct examples where the optimal presentation is exponentially faster than any presentation that minimizes the number of generators; see \cite[\S 5.3]{Sutherland:HilbertClassPolynomials}.

Enumerating $\Ell_\O(\Fp)$ using a polycyclic presentation involves walking along the surfaces of various $\ell$-volcanoes, as in the previous section when testing relations.
But using an optimal presentation will often mean that some of the  primes $\ell_i$ divide~$v$.
This always happens, for example, when $D\equiv 1\bmod 8$, since in this case $\ell_1=2$ divides $v$.
Thus we must be prepared to walk along the surface of an $\ell$-volcano of nonzero depth.
We now give a simple algorithm to do this.
\medskip

\noindent
\textbf{Algorithm} \textsc{WalkSurfacePath}\\
Given $v_0$ on the surface $V_0$ of an $\ell$-volcano of depth $d$ and a positive integer $n < \#V_0$, return a path $v_0,\ldots,v_n$ in $V_0$.
\begin{enumerate}[1.]
\item
If $v_0$ has a single neighbor $v_1$, then return the path $v_0,v_1$.\\
Otherwise, walk a path $v_0,\ldots,v_d$ and set $i\leftarrow 0$.
\item
While $\deg v_{i+d}=1$: replace $v_{i+1},\ldots,v_{i+d}$ by extending the path $v_0,\ldots,v_i$ by~
$d$~steps, starting from an unvisited neighbor $v_{i+1}'$ of $v_i$.
\item
Extend the path $v_0,\ldots,v_{i+d}$ to $v_0,\ldots,v_{i+d+1}$ and increment $i$.
\item
If $i=n$ then return $v_0,\ldots, v_n$; otherwise, go to step 2.
\end{enumerate}
\smallskip

Algorithm \textsc{WalkSurfacePath} requires us to know the depth $d$ of the $\ell$-volcano, which we may determine from the norm equation.
It works by walking an arbitrary path to the floor and then backing up $d$ steps to a vertex that must be on the surface (whenever we leave the surface we must descend to the floor in exactly $d$ steps).
When $d$ or $\ell$ is large, this algorithm is not very inefficient and the pairing-based approach of \cite{Ionica:PairingTheVolcano} may be faster.
But in the context of computing Hilbert class polynomials, both $d$ and $\ell$ are typically quite small.

\begin{remark}[\textbf{Walking the surface with gcds}]
An alternative approach to walking the surface using gcds is given in \cite{EngeSutherland:CRTClassInvariants}.
Suppose we have already enumerated $v_0,\ldots,v_n$ along the surface of an $\ell$-volcano, and have also taken a single step from $v_0$ to an adjacent vertex $v_0'$ on the surface of an $\ell'$-volcano.
We can then compute a path $v_0',\ldots,v_n'$ along the surface of the $\ell$-volcano containing $v_0'$ by computing each $v_{i+1}'$ as the unique root of $f(Y)=\gcd(\Phi_\ell(v_i',Y),\Phi_{\ell'}(v_{i+1},Y))$.
The vertex $v_{i+1}'$ is guaranteed to be on the surface, and the root-finding operation is trivial, since $f(Y)$ has degree 1.
This approach is generally much faster than using either \textsc{WalkSurfacePath} or the algorithm in \cite{Ionica:PairingTheVolcano}, and in practice most of the vertices in $\Ell_\O(\Fp)$ can be enumerated this way; see \cite{EngeSutherland:CRTClassInvariants} for further details.
\end{remark}

\begin{remark}[\textbf{Space complexity}]
A key virtue of the CRT approach is that by using the \emph{explicit CRT} \cite[Thm.\ 3.2]{Bernstein:ModularExponentiation}, it is possible to directly compute the coefficients of $H_D(X)$ modulo an integer $m$ (the characteristic of $\Fq$, for example), without first computing the coefficients over $\Z$.
This means we can compute $H_D(X)$ over $\Fq$ with a space complexity that is quasi-linear in $h(D)\log q$, which may be much smaller than $|D|\log|D|$.
When $h(D)$ is sufficiently composite (often the case), we can use a decomposition of the ring class field to find a root of $H_D(X)$ in~$\Fq$ with a space complexity quasi-linear in $h(D)^{1/2}\log q$; see \cite{Sutherland:CMmethod}.
The low space complexity of the CRT approach has greatly increased the range of feasible discriminants for the CM method: examples with $|D|\approx 10^{16}$ can now be handled \cite{Sutherland:CMmethod}, whereas $|D|\approx 10^{10}$ was previously regarded as a practical upper limit \cite{Enge:FloatingPoint}.
\end{remark}

\subsection{Computing modular polynomials}
All of the algorithms we have discussed depend on modular polynomials $\Phi_\ell(X,Y)$; we even used them to define the graph of $\ell$-isogenies.
We now outline an algorithm to compute $\Phi_\ell$, using the CRT approach described in \cite{BrokerLauterSutherland:CRTModPoly}.
Under the GRH, it runs in $O(\ell^3(\log\ell)^{3+o(1)})$ expected time, which makes it the fastest method known for computing $\Phi_\ell(X,Y)$.
\smallskip

\noindent
\textbf{Algorithm} \textsc{ComputeModularPolynomial}\\
Given an odd prime $\ell$, compute $\Phi_\ell(X,Y)$ as follows:
\begin{enumerate}[1.]
\item Pick an order $\O$ with $h(\O)>\ell+1$ and let $D=\disc(\O)$.
\item Select a sufficiently large set of primes $p$ that satisfy $4p=t^2-\ell^2v^2D$,\\
with $\ell\nmid v$ and $p\equiv 1\bmod \ell$.
\item For each prime $p$, compute $\Phi_\ell(X,Y)\bmod p$ as follows:
\begin{enumerate}[a.]
\item Enumerate $\Ell_\O(\Fp)$ starting from a root $v_0$ of $H_D(X)\bmod p$.
\item Use V\'elu's algorithm to compute a descending $\ell$-isogeny from $v_0$ to $v_0'$.
\item Enumerate $\Ell_{\O'}(\Fp)$ using $v_0'$ as a starting point, where $[\O:\O']=\ell$.
\item Map the $\ell$-volcanoes that make up $\Ell_\O(\Fp)\cup\Ell_{\O'}(\Fp)$.
\item Interpolate $\Phi_\ell(X,Y) \bmod p$.
\end{enumerate}
\item Use the CRT to recover $\Phi_\ell(X,Y)$ over $\Z$ (or over $\Fq$ via the explicit CRT).
\end{enumerate}
\smallskip

The restrictions on $p$ ensure that each element of $\Ell_\O(\Fp)$ lies on the surface of an $\ell$-volcano of depth 1 whose floor consists of elements of $\Ell_{\O'}(\Fp)$.
An example with $\ell=5$ and $D=-151$ is shown below.
\smallskip

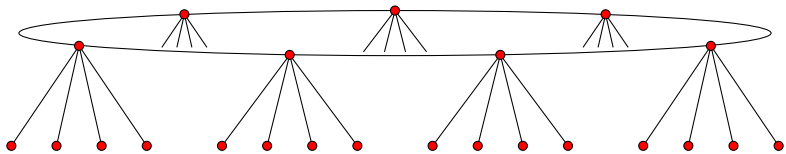
\begin{figure}[htp]
\include{UnmappedVolcano}
\end{figure}
\vspace{-8pt}

When we enumerate $\Ell_\O(\Fp)$ in step 3a, we use a polycyclic presentation $\vec\alpha$ for $\cl(\O)$ derived from prime ideals whose norms are all less than $\ell$ (for $\ell>2$ this is always possible).
By expressing the class $\gamma$ of an invertible $\O$-ideal of norm~$\ell$ in terms of $\vec\alpha$, we can then determine all of the horizontal $\ell$-isogenies between elements of $\Ell_\O(\Fp)$ without knowing $\Phi_\ell$.
In our example with $D=-151$, the presentation~$\vec\alpha$ consists of a single generator $\alpha$ corresponding to an ideal of norm~$2$, with $\gamma=\alpha^3$.  Thus our enumeration of $\Ell_\O(\Fp)$ yields a cycle of $2$-isogenies that we can convert to a cycle of $5$-isogenies by simply picking out every third element.

The application of V\'elu's algorithm in Step 3b involves picking a random point~$P$ of order $\ell$ and computing the $\ell$-isogeny $\varphi$ with $\langle P\rangle$ as its kernel.
This process is greatly facilitated by our choice of $p$, which ensures that $P$ has coordinates in~$\Fp$, rather than an extension field.
We may find that $\varphi$ is a  horizontal $\ell$-isogeny, but we can easily detect this and try again with a different $P$.

As in step 3a, when we enumerate $\Ell_{\O'}(\Fp)$ in step 3c we use a polycyclic presentation $\vec\beta$ for $\cl(\O')$ derived form prime ideals whose norms are all less than~$\ell$.
There are no horizontal $\ell$-isogenies between elements of $\Ell_{\O'}(\Fp)$, but we need to connect each element of $\Ell_{\O'}(\Fp)$ to its $\ell$-isogenous parent in $\Ell_\O(\Fp)$.
This is done by identifying one child $v'$ of each parent  and then identifying that child's siblings, which are precisely the elements of $\Ell_{\O'}(\Fp)$ that are related to $v'$ by a cyclic isogeny of degree $\ell^2$.
By expressing the class $\gamma'$ of an invertible $\O'$ ideal of norm $\ell^2$ in terms of $\vec\beta$, we can identify the $\ell^2$-isogeny cycles of siblings in $\Ell_{\O'}(\Fp)$; these are precisely the cosets of the homomorphism $\rho\colon\cl(\O')\to\cl(\O)$ in \S\ref{subsection:vertical}.

After identifying the horizontal isogenies among the vertices $v$ in $\Ell_{\O}(\Fp)$ and the children of each $v$, we can completely determine the subgraph of $G_\ell(\Fp)$ on $\Ell_\O(\Fp)\cup\Ell_{\O'}(\Fp)$; this is what it means to ``map" the $\ell$-volcanoes in step 3d.
 In our example with $D=-151$ there is just one $\ell$-volcano; the figure below depicts the result of mapping this $\ell$-volcano when $p=4451$.
\smallskip

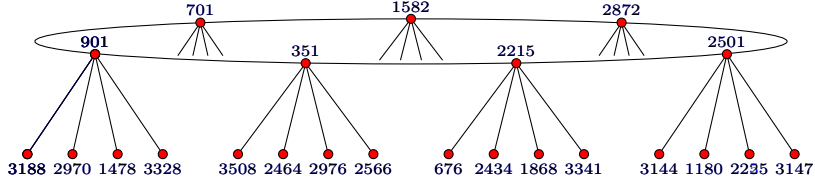
\begin{figure}[htp]
\include{MappedVolcano}
\end{figure}
\vspace{-8pt}

In step 3e we compute, for each of $\ell+2$ vertices $v_i\in\Ell_\O(\Fp)$, the polynomial $\phi_i(Y)=\Phi_\ell(v_i,Y)=\prod_j(Y-v_{ij})$, where $v_{ij}$ ranges over the $\ell+1$ neighbors of $v_i$ in $G_\ell(\Fp)$.  We can then interpolate the coefficients of $\Phi_\ell(X,Y)=\sum_{i,j} c_{ij}X^iY^j$ as follows:
if $\psi_j(X)$ is the unique polynomial of degree at most $\ell+1$ for which $\psi_j(v_i)$ is the coefficient of $Y^j$ in $\phi_i(Y)$, then $c_{ij}$ is the coefficient of $X^i$ in $\psi_j(X)$.

\begin{remark}[\textbf{Weber modular polynomials}]
This algorithm can compute modular polynomials for many modular functions besides the $j$-function; see \cite[\S 7]{BrokerLauterSutherland:CRTModPoly}.
This includes the Weber $f$-function that satisfies $(f(z)^{24}-16)^3 =f(z)^{24}j(z)$.
The modular polynomials $\Phi_\ell^f(X,Y)$ for $f(z)$ are sparser than $\Phi_\ell(X,Y)$ by a factor of~$24$, and have coefficients whose binary representation is smaller by a factor of approximately~$72$.
Thus the total size of $\Phi_\ell^f$ is roughly $1728$ times smaller than $\Phi_\ell$, and it can be computed nearly $1728$ times faster.
\end{remark}

\begin{remark}[\textbf{Modular polynomials of composite level}]
A generalization of this approach that efficiently computes modular polynomials $\Phi_N(X,Y)$ for composite values of $N$ can be found in \cite{BruinierOnoSutherland:PartitionClassPolynomials}.
\end{remark}

\begin{remark}[\textbf{Evaluating modular polynomials}]
Most applications that use $\Phi_\ell(X,Y)$, including all the algorithms we have considered here, only require the instantiated polynomial $\phi(Y)=\Phi_\ell(j(E),Y)$.
A space-efficient algorithm for directly computing $\phi(Y)$ without using $\Phi_\ell(X,Y)$ appears elsewhere in this volume \cite{Sutherland:ModpolyEvaluation}.
\end{remark}

The isogeny volcano algorithm for computing  $\Phi_\ell(X,Y)$ has substantially increased the feasible range of $\ell$:
it is now possible to compute $\Phi_\ell$ with $\ell\approx10,000$, and for $\Phi_\ell^f$ we can handle $\ell\approx 60,000$.
It has also greatly reduced the time required for these computations, as may be seen in the tables of \cite[\S 8]{BrokerLauterSutherland:CRTModPoly}.

\section{Acknowledgements}
I am grateful to Gaetan Bisson for his feedback on an early draft of this article, and to the editors, for their careful review.

\bibliographystyle{amsplain}
%\bibliography{../general}
\input{IsogenyVolcanoes.bbl}

\end{document}

%% file: VolcanoSideView.tex
\begin{tikzpicture}[scale=2.5]
\def \x {-4.5}
\def \backspread {0.08}
\def \frontspreadtop {0.35}
\def \frontspreadbottom {0.23}
\def \vradius {0.03}
\def \backshift {1.4}
\def \lheight {0.6}

\def \ell {3}
\foreach \x in {0}
{
\draw (\x,0) ellipse (1 and 0.1);
\foreach \i in {2,...,\ell}
{
\foreach \off in {-0.7,0.7}
{
\draw ({\x+\off/\backshift},0.1) -- ({\x+\off/\backshift + \backspread * ((2 - \ell) + 2 * (\i - 2))},-0.06);
\draw[fill=red] ({\x+\off/\backshift},0.085) circle (\vradius);
\def \xx {{\x+\off + \frontspreadtop * ((2 - \ell) + 2 * (\i - 2))}}
\draw (\x+\off,-0.1) -- (\xx,-\lheight);
\foreach \j in {1,...,\ell}
{
\def \xxx {{(\xx) + \frontspreadbottom/2 * ((1 - \ell) + 2 * (\j - 1))}}
\draw (\xx,-\lheight) -- (\xxx,-2 * \lheight);
\draw [fill=red] (\xxx,-2*\lheight) circle (\vradius);
}
\draw [fill=red] (\xx,-\lheight) circle (\vradius);
\draw[fill=red] (\x+\off,-0.08) circle (\vradius);
}
}
}
\end{tikzpicture}

%% file: VolcanoTopView.tex
\begin{tikzpicture}[scale=1.5]
\draw (0,0) circle (0.7);
\draw[black] (0,-0.7) -- (-0.7,-1.5);
\draw[black] (0,-0.7) -- (0.7,-1.5);
\draw[black] (0,0.7) -- (-0.7,1.5);
\draw[black] (0,0.7) -- (0.7,1.5);
\draw[black] (-0.7,0) -- (-1.5,-0.7);
\draw[black] (-0.7,0) -- (-1.5,0.7);
\draw[black] (0.7,0) -- (1.5,-0.7);
\draw[black] (0.7,0) -- (1.5,0.7);
\draw[black] (-0.7,-1.5) -- (-0.2,-2.4);
\draw[black] (-0.7,-1.5) -- (-0.7,-2.4);
\draw[black] (-0.7,-1.5) -- (-1.2,-2.4);
\draw[black] (-0.7,1.5) -- (-0.2,2.4);
\draw[black] (-0.7,1.5) -- (-0.7,2.4);
\draw[black] (-0.7,1.5) -- (-1.2,2.4);
\draw[black] (0.7,-1.5) -- (0.2,-2.4);
\draw[black] (0.7,-1.5) -- (0.7,-2.4);
\draw[black] (0.7,-1.5) -- (1.2,-2.4);
\draw[black] (0.7,1.5) -- (0.2,2.4);
\draw[black] (0.7,1.5) -- (0.7,2.4);
\draw[black] (0.7,1.5) -- (1.2,2.4);
\draw[black] (-1.5,-0.7) -- (-2.4,-0.2);
\draw[black] (-1.5,-0.7) -- (-2.4,-0.7);
\draw[black] (-1.5,-0.7) -- (-2.4,-1.2);
\draw[black] (-1.5,0.7) -- (-2.4,0.2);
\draw[black] (-1.5,0.7) -- (-2.4,0.7);
\draw[black] (-1.5,0.7) -- (-2.4,1.2);
\draw[black] (1.5,-0.7) -- (2.4,-0.2);
\draw[black] (1.5,-0.7) -- (2.4,-0.7);
\draw[black] (1.5,-0.7) -- (2.4,-1.2);
\draw[black] (1.5,0.7) -- (2.4,0.2);
\draw[black] (1.5,0.7) -- (2.4,0.7);
\draw[black] (1.5,0.7) -- (2.4,1.2);
\draw[fill=red] (0,0.7) circle (0.05);
\draw[fill=red] (0,-0.7) circle (0.05);
\draw[fill=red] (0.7,0) circle (0.05);
\draw[fill=red] (-0.7,0) circle (0.05);
\draw[fill=red] (-0.7,-1.5) circle (0.05);
\draw[fill=red] (-0.7,1.5) circle (0.05);
\draw[fill=red] (0.7,-1.5) circle (0.05);
\draw[fill=red] (0.7,1.5) circle (0.05);
\draw[fill=red] (-1.5,-0.7) circle (0.05);
\draw[fill=red] (-1.5,0.7) circle (0.05);
\draw[fill=red] (1.5,-0.7) circle (0.05);
\draw[fill=red] (1.5,0.7) circle (0.05);
\draw[fill=red] (-2.4,-0.2) circle (0.05);
\draw[fill=red] (-2.4,-0.7) circle (0.05);
\draw[fill=red] (-2.4,-1.2) circle (0.05);
\draw[fill=red] (-2.4,0.2) circle (0.05);
\draw[fill=red] (-2.4,0.7) circle (0.05);
\draw[fill=red] (-2.4,1.2) circle (0.05);
\draw[fill=red] (2.4,-0.2) circle (0.05);
\draw[fill=red] (2.4,-0.7) circle (0.05);
\draw[fill=red] (2.4,-1.2) circle (0.05);
\draw[fill=red] (2.4,0.2) circle (0.05);
\draw[fill=red] (2.4,0.7) circle (0.05);
\draw[fill=red] (2.4,1.2) circle (0.05);
\draw[fill=red] (-0.2,-2.4) circle (0.05);
\draw[fill=red] (-0.7,-2.4) circle (0.05);
\draw[fill=red] (-1.2,-2.4) circle (0.05);
\draw[fill=red] (0.2,-2.4) circle (0.05);
\draw[fill=red] (0.7,-2.4) circle (0.05);
\draw[fill=red] (1.2,-2.4) circle (0.05);
\draw[fill=red] (-0.2,2.4) circle (0.05);
\draw[fill=red] (-0.7,2.4) circle (0.05);
\draw[fill=red] (-1.2,2.4) circle (0.05);
\draw[fill=red] (0.2,2.4) circle (0.05);
\draw[fill=red] (0.7,2.4) circle (0.05);
\draw[fill=red] (1.2,2.4) circle (0.05);
\end{tikzpicture}

%% file: UnmappedVolcano.tex
\begin{tikzpicture}
\draw (0,0) ellipse (5 and 0.3);
\draw (-2.8,0.25) -- (-3.1,-0.19);
\draw (-2.8,0.25) -- (-2.9,-0.19);
\draw (-2.8,0.25) -- (-2.7,-0.19);
\draw (-2.8,0.25) -- (-2.5,-0.19);
\draw[fill=red] (-2.8,0.25) circle (0.06);
\draw (0,0.3) -- (-0.42,-0.25);
\draw (0,0.3) -- (-0.14,-0.25);
\draw (0,0.3) -- (0.14,-0.25);
\draw (0,0.3) -- (0.42,-0.25);
\draw[fill=red] (0,0.3) circle (0.06);
\draw (2.8,0.25) -- (3.1,-0.19);
\draw (2.8,0.25) -- (2.9,-0.19);
\draw (2.8,0.25) -- (2.7,-0.19);
\draw (2.8,0.25) -- (2.5,-0.19);
\draw[fill=red] (2.8,0.25) circle (0.06);
\draw (-4.2,-0.17) -- (-5.10,-1.5);
\draw[fill=red] (-5.10,-1.5) circle (0.06);
\draw (-4.2,-0.17) -- (-4.50,-1.5);
\draw[fill=red] (-4.50,-1.5) circle (0.06);
\draw (-4.2,-0.17) -- (-3.90,-1.5);
\draw[fill=red] (-3.90,-1.5) circle (0.06);
\draw (-4.2,-0.17) -- (-3.30,-1.5);
\draw[fill=red] (-3.30,-1.5) circle (0.06);
\draw[fill=red] (-4.2,-0.17) circle (0.06);

\draw (-1.4,-0.29) -- (-2.30,-1.5);
\draw[fill=red] (-2.30,-1.5) circle (0.06);
\draw (-1.4,-0.29) -- (-1.70,-1.5);
\draw[fill=red] (-1.70,-1.5) circle (0.06);
\draw (-1.4,-0.29) -- (-1.10,-1.5);
\draw[fill=red] (-1.10,-1.5) circle (0.06);
\draw (-1.4,-0.29) -- (-0.50,-1.5);
\draw[fill=red] (-0.50,-1.5) circle (0.06);
\draw[fill=red] (-1.4,-0.29) circle (0.06);

\draw (1.4,-0.29) -- (2.30,-1.5);
\draw[fill=red] (2.30,-1.5) circle (0.06);
\draw (1.4,-0.29) -- (1.70,-1.5);
\draw[fill=red] (1.70,-1.5) circle (0.06);
\draw (1.4,-0.29) -- (1.10,-1.5);
\draw[fill=red] (1.10,-1.5) circle (0.06);
\draw (1.4,-0.29) -- (0.50,-1.5);
\draw[fill=red] (0.50,-1.5) circle (0.06);
\draw[fill=red] (1.4,-0.29) circle (0.06);

\draw (4.2,-0.17) -- (5.10,-1.5);
\draw[fill=red] (5.10,-1.5) circle (0.06);
\draw (4.2,-0.17) -- (4.50,-1.5);
\draw[fill=red] (4.50,-1.5) circle (0.06);
\draw (4.2,-0.17) -- (3.90,-1.5);
\draw[fill=red] (3.90,-1.5) circle (0.06);
\draw (4.2,-0.17) -- (3.30,-1.5);
\draw[fill=red] (3.30,-1.5) circle (0.06);
\draw[fill=red] (4.2,-0.17) circle (0.06);
\end{tikzpicture}

%% file: MappedVolcano.tex
\begin{tikzpicture}
\fontsize{6pt}{6pt}
\draw (0,0) ellipse (5 and 0.3);
\draw (-2.8,0.25) -- (-3.1,-0.19);
\draw (-2.8,0.25) -- (-2.9,-0.19);
\draw (-2.8,0.25) -- (-2.7,-0.19);
\draw (-2.8,0.25) -- (-2.5,-0.19);
\draw[fill=red] (-2.8,0.25) circle (0.06);
\draw (0,0.3) -- (-0.42,-0.25);
\draw (0,0.3) -- (-0.14,-0.25);
\draw (0,0.3) -- (0.14,-0.25);
\draw (0,0.3) -- (0.42,-0.25);
\draw[fill=red] (0,0.3) circle (0.06);
\draw (2.8,0.25) -- (3.1,-0.19);
\draw (2.8,0.25) -- (2.9,-0.19);
\draw (2.8,0.25) -- (2.7,-0.19);
\draw (2.8,0.25) -- (2.5,-0.19);
\draw[fill=red] (2.8,0.25) circle (0.06);
\draw (-4.2,-0.17) -- (-5.10,-1.5);
\draw[fill=red] (-5.10,-1.5) circle (0.06);
\draw (-4.2,-0.17) -- (-4.50,-1.5);
\draw[fill=red] (-4.50,-1.5) circle (0.06);
\draw (-4.2,-0.17) -- (-3.90,-1.5);
\draw[fill=red] (-3.90,-1.5) circle (0.06);
\draw (-4.2,-0.17) -- (-3.30,-1.5);
\draw[fill=red] (-3.30,-1.5) circle (0.06);
\draw[fill=red] (-4.2,-0.17) circle (0.06);

\draw (-1.4,-0.29) -- (-2.30,-1.5);
\draw[fill=red] (-2.30,-1.5) circle (0.06);
\draw (-1.4,-0.29) -- (-1.70,-1.5);
\draw[fill=red] (-1.70,-1.5) circle (0.06);
\draw (-1.4,-0.29) -- (-1.10,-1.5);
\draw[fill=red] (-1.10,-1.5) circle (0.06);
\draw (-1.4,-0.29) -- (-0.50,-1.5);
\draw[fill=red] (-0.50,-1.5) circle (0.06);
\draw[fill=red] (-1.4,-0.29) circle (0.06);

\draw (1.4,-0.29) -- (2.30,-1.5);
\draw[fill=red] (2.30,-1.5) circle (0.06);
\draw (1.4,-0.29) -- (1.70,-1.5);
\draw[fill=red] (1.70,-1.5) circle (0.06);
\draw (1.4,-0.29) -- (1.10,-1.5);
\draw[fill=red] (1.10,-1.5) circle (0.06);
\draw (1.4,-0.29) -- (0.50,-1.5);
\draw[fill=red] (0.50,-1.5) circle (0.06);
\draw[fill=red] (1.4,-0.29) circle (0.06);

\draw (4.2,-0.17) -- (5.10,-1.5);
\draw[fill=red] (5.10,-1.5) circle (0.06);
\draw (4.2,-0.17) -- (4.50,-1.5);
\draw[fill=red] (4.50,-1.5) circle (0.06);
\draw (4.2,-0.17) -- (3.90,-1.5);
\draw[fill=red] (3.90,-1.5) circle (0.06);
\draw (4.2,-0.17) -- (3.30,-1.5);
\draw[fill=red] (3.30,-1.5) circle (0.06);
\draw[fill=red] (4.2,-0.17) circle (0.06);

\draw (-4.2,0) node{\color{blue}901};
\draw (-4.2,0) node{901};
\draw (-4.2,0) node{\color{blue}901};
\draw (-4.2,0) node{901};
\draw (-1.4,-0.12) node{\color{blue}351};
\draw (-1.4,-0.12) node{351};
\draw (1.4,-0.12) node{\color{blue}2215};
\draw (1.4,-0.12) node{2215};
\draw (4.2,0) node{\color{blue}2501};
\draw (4.2,0) node{2501};
\draw (2.8,0.42) node{\color{blue}2872};
\draw (2.8,0.42) node{2872};
\draw (0,0.46) node{\color{blue}1582};
\draw (0,0.46) node{1582};
\draw (-2.8,0.42) node{\color{blue}701};
\draw (-2.8,0.42) node{701};
\draw[blue] (-4.2,-0.17) -- (-5.10,-1.5);
\draw[fill=red] (-4.2,-0.17) circle (0.06);
\draw[fill=red] (-5.10,-1.5) circle (0.06);
\draw (-5.10,-1.7) node{\color{blue}3188};
\draw (-5.10,-1.7) node{3188};
\draw[black] (-4.2,-0.17) -- (-5.10,-1.5);
\draw[fill=red] (-4.2,-0.17) circle (0.06);
\draw[fill=red] (-5.10,-1.5) circle (0.06);
\draw (-5.10,-1.7) node{\color{blue}3188};
\draw (-4.50,-1.7) node{\color{blue}2970};
\draw (-3.90,-1.7) node{\color{blue}1478};
\draw (-3.30,-1.7) node{\color{blue}3328};
\draw (-5.10,-1.7) node{3188};
\draw (-4.50,-1.7) node{2970};
\draw (-3.90,-1.7) node{1478};
\draw (-3.30,-1.7) node{3328};
\draw (-2.3,-1.7) node{\color{blue}3508};
\draw (-1.7,-1.7) node{\color{blue}2464};
\draw (-1.1,-1.7) node{\color{blue}2976};
\draw (-0.5,-1.7) node{\color{blue}2566};
\draw (-2.3,-1.7) node{3508};
\draw (-1.7,-1.7) node{2464};
\draw (-1.1,-1.7) node{2976};
\draw (-0.5,-1.7) node{2566};
\draw (2.3,-1.7) node{\color{blue}3341};
\draw (1.7,-1.7) node{\color{blue}1868};
\draw (1.1,-1.7) node{\color{blue}2434};
\draw (0.5,-1.7) node{\color{blue}676};
\draw (2.3,-1.7) node{3341};
\draw (1.7,-1.7) node{1868};
\draw (1.1,-1.7) node{2434};
\draw (0.5,-1.7) node{676};
\draw (5.10,-1.7) node{\color{blue}3147};
\draw (4.50,-1.7) node{\color{blue}2255};
\draw (3.90,-1.7) node{\color{blue}1180};
\draw (3.30,-1.7) node{\color{blue}3144};

\draw (5.10,-1.7) node{3147};
\draw (4.50,-1.7) node{2225};
\draw (3.90,-1.7) node{1180};
\draw (3.30,-1.7) node{3144};
\end{tikzpicture}

%% file: IsogenyVolcanoes.bbl
\providecommand{\bysame}{\leavevmode\hbox to3em{\hrulefill}\thinspace}
\providecommand{\MR}{\relax\ifhmode\unskip\space\fi MR }
\providecommand{\MRhref}[2]{%
  \href{http://www.ams.org/mathscinet-getitem?mr=#1}{#2}
}
\providecommand{\href}[2]{#2}